\documentclass[11pt,righttag]{amsart}
 \pdfoutput=1 

\usepackage{graphicx}
\usepackage{wrapfig}
\usepackage{xypic}

\begin{document}

\topmargin-0.1in
\textheight8.5in
\textwidth5.5in
 
\footskip35pt
\oddsidemargin.5in
\evensidemargin.5in

\newcommand{\V}{{\cal V}}      % cal in math mode
\renewcommand{\O}{{\cal O}}
\newcommand{\Ext}{\hbox{Ext}}

\newcommand{\lonto}{{\protect \longrightarrow\!\!\!\!\!\!\!\!\longrightarrow}}

\newcommand{\m}{{\frak m}}
\newcommand{\gl}{{\frak g}{\frak l}}
\newcommand{\ssl}{{\frak s}{\frak l}}

\renewcommand{\d}{\partial}

\newcommand{\ds}{\displaystyle}
\newcommand{\s}{\sigma}
\renewcommand{\l}{\lambda}
\renewcommand{\a}{\alpha}
\renewcommand{\b}{\beta}
\newcommand{\G}{\Gamma}
\newcommand{\g}{\gamma}

\newcommand{\C}{{\Bbb C}}
\newcommand{\F}{{\Bbb F}}
\newcommand{\N}{{\Bbb N}}
\newcommand{\Z}{{\Bbb Z}}
\newcommand{\ZZ}{{\Bbb Z}}
\newcommand{\RR}{{\Bbb R}}
\newcommand{\PP}{{\Bbb P}}
\newcommand{\K}{{\cal K}}
\newcommand{\ve}{{\varepsilon}}
\newcommand{\cupp}{{\star}}

\newcommand{\rowxy}{(x\ y)}
\newcommand{\colxy}{ \left({\begin{array}{c} x \\ y \end{array}}\right)}
\newcommand{\scolxy}{\left({\begin{smallmatrix} x \\ y
\end{smallmatrix}}\right)}

\renewcommand{\P}{{\Bbb P}}

\newcommand{\la}{\langle}
\newcommand{\ra}{\rangle}

\newtheorem{thm}{Theorem}[section]
\newtheorem{lemma}[thm]{Lemma}
\newtheorem{cor}[thm]{Corollary}
\newtheorem{prop}[thm]{Proposition}

\theoremstyle{definition}
\newtheorem{defn}[thm]{Definition}
\newtheorem{notn}[thm]{Notation}
\newtheorem{ex}[thm]{Example}
\newtheorem{rmk}[thm]{Remark}
\newtheorem{rmks}[thm]{Remarks}
\newtheorem{note}[thm]{Note}
\newtheorem{example}[thm]{Example}
\newtheorem{problem}[thm]{Problem}
\newtheorem{ques}[thm]{Question}
 
\numberwithin{equation}{section}%this numbers equations by section

\newcommand{\onto}{{\protect \rightarrow\!\!\!\!\!\rightarrow}}
\newcommand{\donto}{\put(0,-2){$|$}\put(-1.3,-12){$\downarrow$}{\put(-1.3,-14.5) 

{$\downarrow$}}}

\newcounter{letter}
\renewcommand{\theletter}{\rom{(}\alph{letter}\rom{)}}

\newenvironment{lcase}{\begin{list}{~~~~\theletter} {\usecounter{letter}
\setlength{\labelwidth4ex}{\leftmargin6ex}}}{\end{list}}

\newcounter{rnum}
\renewcommand{\thernum}{\rom{(}\roman{rnum}\rom{)}}

\newenvironment{lnum}{\begin{list}{~~~~\thernum}{\usecounter{rnum}
\setlength{\labelwidth4ex}{\leftmargin6ex}}}{\end{list}}

%\begin{document}
 
\title{Noncommutative Koszul Algebras from Combinatorial Topology}


\author[Cassidy, Phan and Shelton]{ }
\maketitle

\begin{center}

\vskip-.2in Thomas Cassidy$^\dagger$, Christopher Phan$^\ddagger$ and Brad Shelton$^\ddagger$\\
\bigskip

$^\dagger$Department of Mathematics\\ Bucknell University\\
Lewisburg, Pennsylvania  17837
\\ \ \\

   $^\ddagger$Department of Mathematics\\ University of Oregon\\
Eugene, Oregon  97403-1222
\end{center}

\setcounter{page}{1}

\thispagestyle{empty}

 \vspace{0.2in}

\begin{abstract}
\baselineskip15pt
Associated to any uniform finite layered graph $\Gamma$ there is a  noncommutative
  graded quadratic algebra $A(\Gamma)$ given by a construction due to Gelfand, Retakh, Serconek and Wilson.   It is natural to ask when these algebras are Koszul.  Unfortunately,  a mistake in the literature states that all such algebras are Koszul.  That is not the case and the theorem was recently retracted.   
We analyze the Koszul property of these algebras for two large classes of graphs associated to finite regular CW complexes, $X$. Our methods are primarily topological. We solve the Koszul problem by  introducing new cohomology groups $H_X(n,k)$, generalizing the usual cohomology groups $H^n(X)$.  Along with several other results, our methods give a new and primarily topological proof of the main result of \cite{SW} and \cite{Piontkovski1}.  
\end{abstract}

\bigskip

\baselineskip18pt

%%%%%%%%%%%%%%%%%%%%%%%%%%%%%%%%%%%%%%%%%%%%%%%%%%%%%%%%%%%%%%%%%%%%%%

\section{Introduction}

In 2001, I. Gelfand, V. Retakh and R.L. Wilson  \cite{GRW}, introduced an interesting family of algebras, $Q_n$, associated to the problem of factoring noncommuting polynomials of degree $n$.  The algebra $Q_n$ is defined using the combinatorics of the poset of all subsets of $\{1\ldots,n\}$, i.e. 
the simplicial complex of a single $n$-simplex. In 2005  these same authors joined by S. Serconek,  \cite{GRSW}, generalized the definition of $Q_n$ so as to obtain a graded algebra $A(\Gamma)$ associated to any layered graph $\Gamma$.   

The algebra $A(\Gamma)$ is defined as follows.  Let $W$ be the vector space spanned by the edges of $\Gamma$ and let $t$ be a central indeterminate over the free algebra $T(W)$.  Then $A(\Gamma)$ is the graded 
algebra obtained by factoring from $T(W)$ the relations obtained by equating coefficients of  $t$
in the expression
$$(t-e_1)(t-e_2)\cdots(t-e_n) = (t-f_1)(t-f_2)\cdots(t-f_n)$$
whenever $e_1,\ldots, e_n$ and $f_1,\dots,f_n$ define directed paths in $\Gamma$ with the same head and tail.  In the recent paper \cite{RSW3}  the algebras $A(\Gamma)$ are given the name
{splitting algebras}.

Many interesting results have been proved about the algebras $A(\Gamma)$.  
It was shown in \cite{GRW} that $Q_n$ is quadratic. In \cite{GGRSW} the Hilbert series of $Q_n$ and its quadratic dual were calculated and observed to satisfy what is known as numerical Koszulity.  Shortly thereafter two papers, \cite{SW} and \cite{Piontkovski1}, gave independent proofs that the algebra $Q_n$ is Koszul.  These proofs are both very combinatorial in nature, the first utilizing the notion of distributive lattices and the second using Groebner basis theory.
In \cite{GRSW} a basis was found for the more general splitting algebras $A(\Gamma)$ and then, in \cite{RSW1}, a beautiful closed formula was given for the Hilbert series of 
$A(\Gamma)$.  

This paper is mainly concerned with the question of when the algebras $A(\Gamma)$ are Koszul.  In \cite{RSW2} it was shown that $A(\Gamma)$ is a quadratic algebra whenever $\Gamma$ satisfies a simple combinatorial property called {\it uniform}.   
Unfortunately, it was also claimed in \cite{RSW2} that $A(\Gamma)$ is Koszul whenever
$\Gamma$ is uniform.  There exist many counterexamples to this claim.  For an example see \cite{RSW3}, which also contains interesting results on the Koszulity of $A(\Gamma)$.

The algebra $A(\Gamma)$ is naturally filtered using a rank function associated to the rank function on $\Gamma$.  The associated graded algebra, $grA(\Gamma)$, has substantially simpler relations than $A(\Gamma)$.  Under the uniformity hypothesis, $grA(\Gamma)$ and $A(\Gamma)$ are quadratic, from which it  follows that one may prove $A(\Gamma)$ is Koszul by instead proving that $grA(\Gamma)$ is Koszul.  

In this paper we work exclusively with the quadratic dual algebra $R(\Gamma) = grA(\Gamma)^!$, as defined in section 3.  Since $R(\Gamma)$ is Koszul if and only if 
$grA(\Gamma)$ is Koszul, any results about the Koszulity of $R(\Gamma)$ give us information about the Koszulity of $A(\Gamma)$.  

The primary goal of this paper is to concentrate on the topological case:   those cases where the layered graph $\Gamma$ is associated to a topological space $X$.  In particular we wish to have a more geometric or topological understanding of why the algebras $Q_n$ are Koszul.   Given a regular CW complex $X$, the face poset $P(X)$ of $X$ becomes a layered graph if we simply attach a new minimal element $\bar 0$ corresponding to the empty set.  We refer to this graph as $\bar P(X)$. 

\begin{thm}\label{main1}(See Theorem \ref{easy})
Let $F$ be any field and $X$ a regular CW complex with face poset $P(X)$.  Then
$R(\bar P(X))$ and $A(\bar P(X))$ are Koszul algebras.
\end{thm}

A special case of this theorem is proved in \cite{RSW3}, where it is shown that
$A(\bar P(X))$ is Koszul whenever $X$ is a finite simplicial complex.  That result also follows from a purely algebraic result in this paper, Corollary \ref{cyclic}, given that the algebras
$Q_n$ and $grQ_n$ are Koszul.  We wish to stress that the techniques used to prove
Theorem \ref{main1} are completely new and based mostly on algebraic topology. In particular we obtain a new and independent proof that
$Q_n$ is Koszul. 

If the regular CW complex $X$ of dimension $d$ is pure and connected by $(d-1)$-faces (see Section 2 for definitions), then there is another important uniform layered graph associated to $X$.  Let $\hat P(X)$ be the face poset of $X$ extended by both a new minimal element, $\bar 0$ and a new maximal element $\bar 1$.  One can
think of $\bar 1$ as being the space of $X$.  The analysis of the Koszul property for the algebras $A(\hat P(X))$ and $R(\hat P(X))$ is substantially more subtle than that needed for Theorem \ref{main1}.

\begin{thm}\label{main2} 
Let $X$ be a pure and regular CW-complex of dimension $d$ that is connected by 
$(d-1)$-faces. Let $F$ be  a field. 

(1)  If $R(\hat P(X))$ is Koszul over the field $F$, then $H^n(X;F) = 0$ for all $0<n<d$.

(2)  If $X$ is 2-dimensional and $H^1(X;F) = 0$, then $A(\hat P(X))$ and $R(\hat P(X))$ are Koszul algebras over $F$. 

(3)  If $X$ is a manifold or a manifold with boundary and $H^n(X;F)=0$ for $0<k<n$, then
then $A(\hat P(X))$ and $R(\hat P(X))$ are Koszul over $F$.

\end{thm}

All the parts of Theorem \ref{main2} are  corollaries of Theorem \ref{hard}, which
gives necessary and sufficient {\it topological} conditions for the Koszulity of the algebra
$R(\hat P(X))$.  These conditions are expressed in terms of a collection of cohomology groups, $H_X(n,k)$ that are defined and studied in Section 4.  We show that 
$H_X(n,0) = H^n(X)$, but the other groups do not seem to be standard. These groups may turn out to be of independent interest.

Theorem \ref{main2} should be compared to the results in \cite{RSW3}.  In that paper it is shown that
if $A(\hat P(X))$ is Koszul (or numerically Koszul) then there is some restriction on the Euler characteristic of $X$.  Part (2) of Theorem \ref{main2} is also obtained there, but only for oriented 2-manifolds.
 There is no analogue in \cite{RSW3}   of Theorem \ref{hard} or of  Theorem \ref{main2} part (3).

It is interesting to observe from Theorem \ref{main2} that the Koszul property of the algebra
$R(\hat P(X))$ is dependent on the field, unlike the algebras $R(\bar P(X))$.    For example, if $X$ is a regular cellular representation of $\RR\PP^2$, then $R(\hat P(X))$ is Koszul if and only if the field of definition does not have characteristic 2. 

It is natural to ask whether the converse to part (1) of Theorem \ref{main2} holds.  We show by example in Theorem \ref{theexample}, using the topological tools of Theorem \ref{hard},  that the converse is not true.  

In addition to the topological results listed above, we prove some strong algebraic results about the Koszul property of any $R(\Gamma)$.  These results appear in Section 3.  The main result of that section is a necessary and sufficient condition for the Koszulity of 
$R(\Gamma)$ given in terms of explicit formulas for the annihilators of certain elements of $R(\Gamma)$.  These formulas then lead to the introduction of a series of cochain complexes $R_\Gamma(n,k)$ which are shown, in Section 5, to be closely related to the cohomology groups $H_X(n,k)$. One easy but important corollary to the results of Section 3 should be mentioned.  For any vertex $x$ in the layered graph $\Gamma$, let $\Gamma_x$ be the layered subgraph of all vertices $y$ below $x$ (including $x$). 

\begin{thm}\label{algebra1}(See Corollary \ref{cyclic})  Let $\Gamma$ be a  uniform layered graph  Then 
$R(\Gamma)$ is Koszul if and only if $R(\Gamma_x)$ is Koszul for every vertex $x$. 

\end{thm}

Section 2 of this paper contains background definitions and simple results about posets attached to regular CW complexes. Section 3 proves our main algebraic results. Section 4 introduces the new cohomology groups $H_X(n,k)$ associated to a regular CW complex.  We prove all of our main theorems in Section 5. 

The authors would like to heartily thank Hal Sadofsky for many useful discussions and especially for helping us find a preliminary version of Example \ref{singular}.
 
\section{Layered Graphs and Posets}\label{posets}

A {\it layered graph} is a graph $\Gamma$ with a set of vertices
$V= V(\Gamma)= \bigcup_{i=0}^d V_i$ and a set of directed edges $E = E(\Gamma) = 
\bigcup_{i=1}^d E_i$ satisfying the following:

(a) $V_0 = \{ \bar 0 \}$

(b) If $e:a\to b$ is an edge and $e\in E_i$, then $a \in V_i$ and $b\in V_{i-1}$.

(c) There is at most one edge from any $A \in V_i$ to any $B\in V_{i-1}$.

(d) $\bar 0$ is the unique sink in $\Gamma$.   

For our purposes, it is very convenient to use the language of ranked posets.
A finite poset $(P,<)$ is said to be {\it ranked} if $P$ has unique minimal element $\bar 0$ and for any $x$, every maximal chain in the closed interval $[\bar 0,x]$ has the same length.  The common length 
of maximal chains in $[\bar 0,x]$ is then called the rank of $x$, $rk(x)$.

It is clear that the notion of layered graph is the same as the notion of ranked poset, where edges are identified with pairs $(b,a)$ such that $\{a<b\}$ is a maximal chain in $[a,b]$.  We will use the notions of layered graph and ranked poset interchangeably.  

Given a layered graph $\Gamma$ and vertex $a$, we write $\Gamma_a$ for $[\bar 0,a]$.  
For any $0 \le n\le rk(a)$ let
$S_a(n) = \{b  <a \,|\, rk(a)-rk(b) = n\}$.  
There is an equivalence relation on the set $S_a(1)$ defined by extending transitively the
relation: $b \sim c$ if $b=c$ or $S_b(1)\cap S_c(1) \ne \emptyset$.  We  denote this
equivalence relation as $\sim_a$.   We can now recall the following important definition from 
\cite{RSW2}.

\begin{defn}
The layered graph $\Gamma$ is {\it uniform} if for each $a$, the equivalence relation
$\sim_a$ on $S_a(1)$ has only one equivalence class.  
\end{defn}

It is clear that $\Gamma$ is uniform if and only if $\Gamma_x$ is uniform for every 
$x \in \Gamma$.  We need to extend the idea behind uniform to lower layers of the graph.

\begin{defn}  Let $\Gamma$ be a layered graph.  Let $a$, $a'$ be two vertices
of rank $k$. 

A {\it down-up sequence} from $a$ to $a'$ in $\Gamma$ is a sequence of elements $a_0=a$, $a_1,\cdots, a_n=a'$ in $\Gamma$, of rank $k$,
and another sequence of elements $b_1,b_2,\ldots,b_n$ of rank $k-1$ such that for all 
$i$, $b_i<a_{i-1}$ and $b_i<a_i$.   We allow the case $n=0$, i.e. $a = a'$
 
 An {\it up-down sequence} from $a$ to $a'$ in $\Gamma$ is a sequence of elements $a_0=a$, $a_1,\cdots, a_n=a'$ of rank $k$
and another sequence of elements $b_1,b_2,\ldots,b_n$ in $\Gamma$, of rank $k+1$ such that for all $i$,
$a_{i-1}<b_i$ and $a_i<b_i$.  Again, we allow the case $n=0$ when $a = a'$. 
   
\end{defn}

One sees immediately that $\Gamma$ is uniform if and only if, for any $a,a' \in S_x(1)$, there is a
down-up sequence from $a$ to $a'$ in $\Gamma_x$.  This is then the base case for a simple inductive proof of the following lemma.

\begin{lemma}\label{updown}
Let $\Gamma$ be a uniform graph of rank $d$.  Fix  $x\in V_d$.  

(1)  For any $a,a'\in V_k(\Gamma_x)$ there is a down-up path from $a$ to
$a'$ in $\Gamma_x$. 

(2) For any $a,a' \in V_k(\Gamma_x)$ there is an up-down path from $a$ to $a'$ in 
$\Gamma_x$. 
\end{lemma}

We need one final combinatorial definition.

\begin{defn}(\cite{handbook}, \cite{Bjorner1})
The layered graph $\Gamma$ is {\it thin} if for each pair of vertices 
$a$ and $b$ for which $b \in S_a(2)$, $[b,a]$ has exactly four elements.   I.e. for any two vertices $a$ and $b$ such that 
$rk(a)-rk(b)=2$ there are either
no paths from $a$ to $b$ or exactly 2 paths from $a$ to $b$. 
\end{defn}

We now turn to topological examples.  We begin by recalling some definitions.  In a CW-complex $X$, we denote an open $k$-cell as $e_\alpha^k$ or simply $e_\alpha$ and the associated attaching
map on the closed $k$-disc by $\Phi_\alpha:D_\alpha^k \to X$.  We will only ever consider {\it finite} CW complexes.

\begin{defn}
Let $X$ be a finite CW-complex of dimension $d$.  Then $X$ is {\it regular} if for each $\alpha$ the attaching map $\Phi_\alpha:D_\alpha^n \to X$ is an embedding.  We say $X$ is {\it pure} if $X$ is the union of all of its (closed) $d$-cells.  The underlying topological space of $X$ will be denoted
$||X||$. 
\end{defn}

In a regular CW-complex $X$ we may, and will, identify any closed $n$-cell 
$D_\alpha^n$ with its image in $X$.  The interior of this cell is then $e_\alpha^n$ and the boundary is a homeomorphic image of $S^{n-1}$.   Suppose that $X$ is pure of dimension $d$.  On the set of all
$d$-cells of $X$ define $\alpha \sim \beta$ if $\alpha = \beta$ or if
$\partial D_\alpha^d$ and $\partial D_\beta^d$ share a common $(d-1)$-cell. We extend $\sim$ transitively to an equivalence relation on the $d$-cells of $X$.

\begin{defn}
A pure regular CW complex of dimension $d$ is {\it connected by $(d-1)$-faces} if there is a unique equivalence class of $d$-cells under $\sim$. 

\end{defn}

Clearly a complex that is connected by $(d-1)$-faces  is connected (even when $d=0$.)

We obtain various layered graphs, or ranked posets from the face poset of a regular CW complex $X$.  Let $P(X)$ be the face poset of the regular CW complex $X$, where $\alpha <\beta $ if and only if $e_\alpha \subset D_\beta$. 
Let $\bar P(X)$ be the extended poset $\bar P(X) = P(X)\cup \{\bar 0\}$ with $\bar 0 < \alpha$ for all 
$\alpha \in P(X)$.  Similarly let $\hat P(X) = P(X)\cup \{\bar 0, \bar 1\}$, where $\bar 0 < \alpha < \bar 1$ for all $\alpha\in P(X)$. 

Purity of a CW complex is the hypothesis required to know that $\hat P(X)$ is a ranked poset.  
Connected by $(d-1)$-faces  is the hypothesis required to know that $\hat P(X)$ is uniform.

The following results are standard and can be found in \cite{Bjorner1}, 
\cite{handbook} and \cite{massey}.

\begin{lemma}  Let $X$ be a regular CW-complex of dimension $d$.

(1) If $||X|| = S^d$ then $X$ is pure and connected by $(d-1)$-faces.  Moreover, $\bar P(X)$ is a ranked poset with $rk(\alpha) = dim(e_\alpha)+1$ and $\bar P(X)$ is uniform and thin.
 
(2) For any $X$,  $\bar P(X)$ is a thin ranked poset.  

(3).  If $X$ is pure, then $\hat P(X)$ is a ranked poset.  If $X$ is also connected by $(d-1)$-faces, then 
$\hat P(X)$ is uniform. 
\end{lemma}

We remark that in part (3) above, $\hat P(X)$ may not be thin, although its 
sub-poset $\bar P(X)$ will always be thin.  For example, if $X$ is the simplicial complex obtained by gluing three triangular 2-cells along a common edge, then $\hat P(X)$ is immediately seen to not be thin.  

Let $\alpha$ and $\beta$ index cells in our CW-complex $X$.  By a path from
$\beta$ to $\alpha$ we mean any maximal chain in the interval $[\alpha, \beta]$.  Let 
$\Pi(\beta,\alpha)$ be the set of all such paths.  Given
$\pi_1 = \{\alpha = \alpha_0<\alpha_1< \cdots < \alpha_k=\beta\}$ and 
$\pi_2 = \{\alpha = \gamma_0<\gamma_1< \cdots < \gamma_k=\beta\}$, two paths in
$\Pi(\beta,\alpha)$, we say that $\pi_1$ and $\pi_2$ are diamond equivalent, $\pi_1\diamond \pi_2$, if
either $\pi_1 = \pi_2$ or there is a unique $j$ for which $\alpha_j \ne \gamma_j$.  This is not 
an equivalence relation, but we extend $\diamond$ transitively to an equivalence relation on 
$\Pi(\beta,\alpha)$.

\begin{lemma}\label{diamond}
If $X$ is a regular CW complex, then for any cells $\alpha < \beta$, $\Pi(\beta,\alpha)$ consists of a single $\diamond$-equivalence class. 
\end{lemma}

\begin{proof}
Since we are working in the poset $[\alpha,\beta] \subset \bar P(X)$, we may assume $X = X_\beta$.  Assume $\beta$ is a $d$-cell and $\alpha$ a $k$ cell.  We proceed by 
induction on $d-k$.  There is nothing to prove if $d-k=0$ or $d-k = 1$.   If $d-k = 2$, then the result follows from the fact that $\bar P(X)$ is thin.  For the remainder we assume 
$d-k \ge 3$.  

Let $[\pi_1], [\pi_2],\ldots, [\pi_r]$ be the distinct equivalence classes of $\diamond$-equivalence
on $\Pi(\beta,\alpha)$.  Suppose $r>1$.  For each $[\pi]$, let 
$E([\pi]) = \{\gamma \in (\alpha,\beta) | \gamma \in \pi'$ for some $\pi' \in [\pi]\}$.  By induction, the sets $E([\pi_i])$ must be disjoint, and furthermore, for any $i \ne j$ we can not find $\gamma \in E([\pi_i])$ and $\mu \in E([\pi_j])$ such that $\mu<\gamma$.  
Since the open interval $(\alpha,\beta)$ is the union of the sets $E([\pi_i])$, this shows that
$(\alpha,\beta)$ is not connected as a graph.  

Choose a point $x$ in $e_\alpha \subset X = X_\beta$.  Since $\partial X$ is a $d-1$-sphere,
there is an open neighborhood of $x$ in $\partial X$ homeomorphic to a $d-1$-ball. We may choose this 
ball $B$ small enough to ensure that $B \cap e_\gamma = \emptyset$ for all $\gamma \not\in [\alpha,\beta)$.   Since the codimension of $e_\alpha$ is at least 2 in $B$, the space $B \setminus e_\alpha$ must be connected.  But this ensures that the graph $(\alpha,\beta)$  is connected, contradicting the assumption of the previous paragraph.  This proves the lemma.
\end{proof}

Note: the fact that the open interval $(\alpha,\beta)$ is connected as a 
graph, for $dim(\beta)-\dim(\alpha) >2$, seems to be standard in the literature. 
However, we do not have a reference for it.

\section{Algebraic results.}\label{algebra}

Let $F$ be a field.  Let $\Gamma$ be a uniform ranked poset with unique minimal element $\bar 0$.  It is convenient to let $\Gamma' = \Gamma \setminus \{\bar 0\}$.
We begin by giving an explicit presentation of the quadratic dual algebra, $R(\Gamma)$, of the associated graded algebra of $A(\Gamma)$.  The reader is referred to 
\cite{RSW2} for the proof that $A(\Gamma)$ is quadratic and for the description
of $grA(\Gamma)$. 

Let $W$ be the $F$-vector space with basis
$\{r_x | x \in \Gamma' \}$. We extend the notion of rank to $W$, that is
$rk(v) = n$ if $v$ is a linear combination of elements $r_x$ all of which have rank $n$. For any $i\ge 0$, let
$$r_x(n) = \sum_{y\in S_x(n)} r_y \in W.$$
In particular $r_x(0) = r_x$.   We note that $rk(r_x(n)) = rk(x) -n$ for any $0\le n \le rk(x)-1$. 
It is convenient to set $r_x(n) = 0$ for any $n\ge rk(x)$. 
Let $I_2 = I_2(\Gamma)$ be the subspace of
$W \otimes W$ spanned by the elements 
$$\{r_x \otimes r_y | x,y\in \Gamma', y \not\in S_x(1)\} \,\cup\, \{r_x \otimes r_x(1) | x\in \Gamma'\}.$$
Let $I = \bigoplus_n I_n$ be the graded ideal of $T(W)$ generated by $I_2$.  Then $R = R(\Gamma) :=
T(W)/I(\Gamma)$.  

For any $d\ge 2$, $I_d = I_2 \otimes W^{\otimes d-2} + W \otimes I_{d-1}$.  Suppose  $\sum_x r_x \otimes u_x \in I_d$.  We decompose this as the sum of two terms, 
$\sum_x r_x\otimes v_x \in I_2 \otimes W^{\otimes d-2}$ and  
$\sum_x r_x\otimes w_x \in W \otimes I_{d-1}$.     Clearly then $r_x\otimes w_x \in I_d$ for all $x$.
Since the relations of $R(\Gamma)$ are all of the form $r_y\otimes l_{i,y}$ for some linear terms
$l_{i,y} \in W$, we see that $\sum_x r_x\otimes v_x = \sum_x (r_x \otimes \sum_i \alpha_{i,x} l_{i,x} \otimes z_x)$ for some coefficients $\alpha_{i,x}$ and $z_{i,x} \in W^{\otimes d-2}$.  But this shows
$r_x \otimes v_x \in I_d$ for each $x$ and thus we conclude $r_x\otimes u_x \in I_d$ for each $x$.
The  following innocuous looking lemma follows immediately from these observations. This lemma is a powerful tool.

\begin{lemma}\label{rtideal} 
Let $\Gamma$ be a  uniform ranked poset and $R = R(\Gamma)$. Then 
$$R_+ = \bigoplus_{x\in \Gamma'} r_x R.$$

\end{lemma}

Another consequence of the argument
used to prove \ref{rtideal} is the following lemma.

\begin{lemma}\label{subring}
Let $\Gamma$ be a uniform layered graph.  Then for any 
$x\in \Gamma'$, the natural algebra homomorphism $R(\Gamma_x) \to R(\Gamma)$
is injective. 
\end{lemma}

Now we can give a simple criterion for the Koszul property in terms of annihilators of specific elements of $R(\Gamma)$.  We require some notation.  
For any $x\in \Gamma'$  and $n\ge 0$ we define: 
$ H_x(n) = \sum_{y\not\in S_x(n)} r_yR$. 
By \ref{rtideal}, the sum defining $H_x(n)$ is direct.

We recall that a module $M$ is {\it Koszul} if $M$ admits a linear projective resolution (cf. \cite{PP}). It is clear that any direct summand of a Koszul module is Koszul.

\begin{thm}\label{formula} 
Let $\Gamma$ be a uniform layered graph.  Then $R = R(\Gamma)$ is Koszul if and only if  for every $x \in \Gamma'$ and $0\le n\le rk(x)$:
$$rann_R(r_x(n)) = r_x(n+1)R \oplus H_x(n+1).$$

\end{thm}

\begin{proof}
Let $J_x(k) = r_x(k)R \oplus H_x(k)$.  The sum is direct by \ref{rtideal}.  For any 
$y\in S_x(n)$, $S_y(1) \subset S_x(n+1)$.  It follows that $r_x(n)r_x(n+1) =0$. Similarly, $r_x(n)r_z = 0$ for every $z\not\in S_x(n+1)$. This shows that the right hand side of the formula above is always contained in the left hand side.  

Assume $R$ is Koszul.  Let $F_R$ denote the trivial right $R$ module $R/R_+$.  Then $F_R$ must be a Koszul module.  We prove the formula $rann_R(r_x(k)) = J_x(k+1)$ by induction on $k$.  This is easiest to do if we simultaneously prove that the module 
$r_x(k)R$ is a Koszul module. 

Since $r_xR$ is a direct summand of the Koszul module $R_+$, $r_xR$ is a Koszul module for every 
$x$.  In particular, the right annihilator  of $r_x$ must be linearly generated.  But $J_x(1)$ is exactly the right ideal generated by the linear annihilators of $r_x = r_x(0)$.  This proves the base case, $k=0$.

Now we assume our claim for $k-1$.  Then for each $x$ a short exact sequence:
$$ 0 \to J_x(k) \to R \to r_x(k-1)R \to 0.$$
Since $r_x(k-1)R$ is a Koszul module,  $J_x(k)$ must also be a Koszul module and hence its direct summand $r_x(k)R$ is a Koszul module.  Hence the right annihilator of $r_x(k)$ must be linearly generated.  It suffices thus to prove that every linear right annihilator of $r_x(k)$ is in $J_x(k+1)$.

Let $p \in R_1$ and suppose $r_x(k)p=0$.  We must show $p\in J_x(k+1)$. Write
 $p = \sum_{y\in S_x(k+1)}\alpha_y r_y + \sum_{z\not\in S_x(k+1)} \beta_z r_z$.  Since $r_z \in J_x(k+1)$ for all $z \not\in S_x(k+1)$, we may assume all the coefficients $\beta_z$ are $0$.  Choose
 $y$ and $y'$ in $S_x(k+1)$.   We claim $\alpha_y = \alpha_{y'}$ and this claim will then prove
 that $p$ is a scalar multiple of $r_x(k+1)$, as required. 
 
  By \ref{updown}, we can choose an up-down sequence from $y$ to $y'$ in $\Gamma_x$, say $y =y_0, y_1,\ldots ,y_n = y'$ and $b_1,\ldots ,b_n$.  Since this is an up-down sequence,
 the elements $b_i$ are all in $S_x(k)$.   Since $r_x(k)p=0$, we have by \ref{rtideal}, $r_{b_i}p=0$ for 
 every $i$.  By the $k=0$ case of the theorem,  since $y_{i-1},y_i \in S_{b_i}(1)$, we must have
 $\alpha_{y_{i-1}} = \alpha_{y_i}$ for every $i$.  Hence $\alpha_y = \alpha_{y'}$, as required.  This completes the induction.
 
 Conversely, assume the formula $rann_R(r_x(k)) = J_x(k+1)$ for all $x$ and $k$.   This formula, together with the strong right ideal property of \ref{rtideal} assures us that we can build an (infinite) linear projective resolution of the right trivial module $F_R$. Hence $R$ is Koszul.
\end{proof}

\begin{rmk}\label{basecase}
Since $r_x(0) = r_x$, the formula $rann_R(r_x(0)) = r_x(1)R\oplus H_x(1)$ is true regardless of whether or not $R(\Gamma)$ is Koszul.
\end{rmk}

Combining Theorem \ref{formula} with Lemma \ref{subring} gives the following useful corollary. 

\begin{cor}\label{cyclic}
If $\Gamma$ is  uniform, then $R(\Gamma)$ is Koszul if and only if 
$R(\Gamma_x)$ is Koszul for every $x$ in $\Gamma'$
\end{cor}

We would like to have an inductive tool for proving that $R(\Gamma)$ is Koszul.  According to  \ref{cyclic}, it suffices to consider the case when $\Gamma$ is an
interval, i.e. $\Gamma = \Gamma_x$ for some $x$.  An inductive tool would be one
which allowed us to prove that $R(\Gamma_x)$ is Koszul using the assumption
that $R(\Gamma_y)$ is Koszul for every $y \in (\bar 0,x)$.  
Our last algebraic result of this section will recode the information contained in 
\ref{formula} in terms of cohomology groups to give us just such a tool.   This requires some definitions and a preliminary result.   

\begin{defn} Let
$\Gamma$ be uniform of rank $d+1$. 
Define $d_\Gamma = \sum_{y\in \Gamma'} r_y\in R(\Gamma)$.  We also use 
$d_\Gamma$ to denote the function $d_\Gamma:R(\Gamma) \to R(\Gamma)$ given by left multiplication by $d_\Gamma$.  

For $0\le k\le n\le d$, set $R(n,k) = \sum_{rk(y) = n+1} r_y R_{n-k}$.
\end{defn}

Note that $R_{n-k}$ refers to the $n-k$ graded piece of $R$ and that for any $y$, 
$d_\Gamma = r_y(1) + \sum_{z\not\in S_y(1)} r_z$.  
Therefore $r_y d_\Gamma = 0$ for all $y$ and in particular $d_\Gamma^2 =0$.  For 
any $r_xs \in R(n,k)$, $r_zr_xs = 0$ unless $rk(z) = n+2$, in which case 
$r_zr_xs\in R(n+1,k)$.  In particular $d_\Gamma R(n,k) \subset R(n+1,k)$
(and $d_\Gamma R(d,k) = 0$). Thus, for each $0\le k\le d$ we have cochain complex:
$$ 0\to R(k,k) \buildrel d_\Gamma \over \to \cdots \buildrel d_\Gamma \over \to R(n,k) 
\buildrel d_\Gamma \over \to R(n+1,k) \buildrel d_\Gamma \over \to \ldots
\buildrel d_\Gamma \over \to R(d,k) \buildrel d_\Gamma \over \to 0 .$$

One could easily combine all of these cochain complexes into a single complex represented by $d_\Gamma: R_+ \to R_+$.  Unfortunately, doing so would 
only obfuscate the basic facts we need. 

\begin{lemma}\label{easycohom}
Assume $\Gamma$ is uniform and $\Gamma=\Gamma_x$ for some
$x$ of rank $d+1$. 

(1) $H^k(R(\cdot,k),d_\Gamma) = F$ for all $0\le k\le d$.

(2) $H^d(R(\cdot,k),d_\Gamma) = 0$ for all $0\le k<d$.

(3) $H^{d-1}(R(\cdot,k),d_\Gamma) = 0$ for all $0\le k<d-1$.
\end{lemma}

\begin{proof}
Since $\Gamma=\Gamma_x$,  $\{y \in \Gamma' | rk(y) = n+1 \} = S_x(d-n)$.  Thus we have
$d_\Gamma = \sum_{k=0}^d r_x(k)$. Moreover,
$d_\Gamma R(n,k) = r_x(d-n-1)R(n,k)$.

Using this last formula, we see $H^k(R(\cdot,k),d_\Gamma)$ is the
cohomology of 
$$ 0\to R(k,k) \buildrel r_x(d-k-1)\cdot \over \longrightarrow R(k+1,k).$$
The elements of $R(k,k)$ are linear and have rank $k+1$.  $r_x(d-k-1)$ has rank $k+2$.  Thus the uniformity of $\Gamma$, as shown in the proof of \ref{formula}, implies
that $rann_R(r_x(d-k-1))\cap R(k,k) = F\cdot r_x(d-k)$. Thus $H^k(R(\cdot,k),d_\Gamma)=F$.

Similarly, $H^d(R(\cdot,k),d_\Gamma)$ is the cohomology of:
$$ R(d-1,k) \buildrel r_x\over \to  R(d,k) \to 0.$$
For $k<d$, 
$R(d,k) = r_xR_{d-k} = r_x\sum_y r_y R_{d-k-1} = 
r_xR(d-1,k)$ since $r_xr_y = 0$ unless $rk(y) = d$.   Thus, for $k<d$, 
$H^d(R(\cdot,k),d_\Gamma)=0$.
 
Finally, for $k<d-1$, $H^{d-1}(R(\cdot,k),d_\Gamma)$ is the cohomology of:
$$ R(d-2,k) \buildrel r_x(1)\over \to  R(d-1,k) \buildrel r_x \over \to R(d,k).$$
But $rann(r_x) \cap R(d-1,k) = r_x(1)R(d-2,k)$ as in remark \ref{basecase}.
Thus  $H^{d-1}(R(\cdot,k),d_\Gamma)=0$
\end{proof}

Here now is our inductive tool.

\begin{thm}\label{R-cohomology}
Assume $\Gamma$ is uniform and $\Gamma=\Gamma_x$ for some
$x$ of rank $d+1$. Assume further $R(\Gamma_y)$ is Koszul for
every $y \in (\bar 0,x)$.  Then the algebra $R=R(\Gamma)$ is Koszul
if and only if for each $0\le k\le d$, $H^n(R(\cdot,k),d_\Gamma) = 0$ for $n\ne k$ and $H^k(R(\cdot,k),d_\Gamma) = F$.
\end{thm}

\begin{proof} 
 Assume $R$ is Koszul.  By \ref{easycohom} we may fix $k$ and $n$ with $0\le k < n< d-1$.  Then $H^n(R(\cdot,k),d_\Gamma)$ is the cohomology of 
$$ R(n-1,k)\buildrel r_x(d-n)\over \longrightarrow R(n,k) 
\buildrel r_x(d-n-1)\over \longrightarrow R(n+1,k).$$
Since $\{y \in \Gamma' | rk(y) = n+1 \} = S_x(d-n)$, $H_x(d-n) \cap R(n,k) = 0$
and $r_x(d-n)R(n-1,k) \subset R(n,k)$.  Thus, since $R$ is Koszul,
we have by \ref{formula},
$rann_R(r_x(d-n-1)) \cap R(n,k) = r_x(d-n)R(n-1,k)$. This proves that
$H^n(R(\cdot,k),d_\Gamma)=0$.

Conversely, assume the cohomology groups $H^n(R(\cdot,k),d_\Gamma)$ are as required.
Then the calculations above can be reversed to get the formulas for 
$rann_R(r_x(d-n-1))$ required by \ref{formula}.  The formulas required for all other
$r_y(n)$ are assured by the hypothesis that $R(\Gamma_y)$ is 
Koszul.
\end{proof}

\begin{rmk}
As stated earlier, we have expressed 3.8 in terms that will be useful for induction.  But with
only slightly more work one can show that  Theorem \ref{R-cohomology} is true without the inductive hypotheses or
the hypothesis that $\Gamma=\Gamma_x$.  That is,
for any uniform $\Gamma$ of rank $d+1$, $R(\Gamma)$ is Koszul if and only if
$H^n(R(\cdot,k),d_\Gamma) = 0$ for all $0\le k<n\le d$.  
\end{rmk}

\section{Some cohomology groups associated to regular CW-complexes}\label{topology}

Let $X$ be a regular CW-complex of dimension $d$.   Given a $k$-cell $e_\alpha$ in $X$ we 
let $X_\alpha$ be the closed subcomplex of all cells $e_\beta$ contained in $D_\alpha$.  We also let 
$Y_\alpha$ be the closed subcomplex of all cells $e_\gamma$ such that $e_\alpha$ is not contained in $D_\gamma$.  In particular, $||X_\alpha|| = D_\alpha$, is a $k$-disk.  We note that $\partial e_\alpha \subset Y_\alpha$ and since $X$ is regular, $Y_\alpha$ is not empty unless $X$ is a single point.

We will assume that an orientation has been chosen
for every cell $e_\alpha$ in $X$.  Given an $n$-cell $e_\alpha$ and an $n-1$ cell $e_\beta$ contained
in $\partial e_\alpha$, we define
$d(\alpha,\beta)$ to be $+1$ (respectively $-1$) if the orientation on $e_\beta$ is the same
as (respectively the opposite of) the orientation inherited from $\partial e_\alpha$.  If the $n-1$ cell
$e_\beta$ is not contained in the boundary of $e_\alpha$ then we put $d(\alpha,\beta) = 0$.  Of course $d(\alpha,\beta)$ can also be described as the degree of the appropriate attaching map. 

Let $C^n(X)$ (resp. $C_n(X)$) be the free Abelian group on the symbols $v(\alpha)$ (resp. $w(\alpha)$) for all $n$-cells $e_\alpha$ of $X$.  Then the standard cellular cochain and chain complexes for $X$ are realized on the spaces
$C^n(X)$ and $C_n(X)$ respectively via the differentials:
$$\begin{array}{ccc}
d_X: C^n(X) \to C^{n+1}(X) & \hbox{given by} & 
d_X(v(\alpha)) = \sum_{\gamma }d(\gamma,\alpha)v(\gamma).
\end{array}$$
$$\begin{array}{ccc}
\hat d_X: C_n(X) \to C_{n-1}(X) & \hbox{given by} & 
\hat d_X(w(\alpha)) = \sum_{\gamma }d(\alpha,\gamma)w(\gamma).
\end{array}$$

Given a subcomplex $Y$ of $X$ we similarly write $C^n(X,Y)$ and $C_n(X,Y)$ for the relative cellular cochain and chain complexes of $X$ relative to $Y$ (with the same notation for their differentials $d_X$ and $\hat d_X$).

For $0\le k\le n \le d$, let $C_X(n,k)$ be the free Abelian group on symbols 
$v(\alpha,\beta)$ for all those pairs $(\alpha,\beta)$ such that $e_\alpha$ is an $n$-cell, $e_\beta$ is a $k$-cell and $\beta<\alpha$, i.e.
$e_\beta \subset D_\alpha$.  If $e_\beta$ is not contained in $D_\alpha$,  it is convenient to let the symbol $v(\alpha,\beta)$ be $0$ in $C_X(n,k)$.   We define linear maps $d_X$ and $\hat d_X$ as follows:
$$\begin{array}{ccc}
d_X: C_X(n,k) \to C_X(n+1,k) & \hbox{via} & 
d_X(v(\alpha,\beta)) = \sum_{\gamma }d(\gamma,\alpha)v(\gamma,\beta)
\end{array}$$
$$\begin{array}{ccc}
\hat d_X: C_X(n,k) \to C_X(n,k-1) & \hbox{via} &
\hat  d_X(v(\alpha,\beta)) = \sum_{\gamma }d(\beta,\gamma)v(\alpha,\gamma)
\end{array}$$
It is clear that $d_X\hat d_X = \hat d_X d_X$ since $d_X$ and $\hat d_X$ are transposes of one another. 

For a fixed $k$-cell $e_\alpha$, let $C_X(n,k)_\alpha$ be the subgroup of $C_X(n,k)$ generated by the symbols $v(\beta,\alpha)$ as $\beta$ varies over all $n$-cells for which 
$e_\alpha\subset D_\beta$.  Then
$C_X(n,k) = \oplus_\alpha C_X(n,k)_\alpha$ and
$d_X$ maps $C_X(n,k)_\alpha$ into $C_X(n+1,k)_\alpha$.  But we can identify 
$(C_X(n,k)_\alpha,d_X)$ with $(C^n(X,Y_\alpha),d_X)$ via 
$v(\beta,\alpha) \mapsto v(\beta)$.  We conclude that $(C_X(n,k),d_X)$ is a cochain complex and 
$$H^n(C_X(\cdot,k)_\alpha,d_X) = H^n(X,Y_\alpha).$$

Similarly, for a fixed $n$-cell $\beta$ let $C_X(n,k)^\beta$ be the subroup spanned by
$v(\beta,\alpha)$ as $\alpha$ varies over the $k$-cells contained in $D_\beta$.  
Then we can
identify $(C_X(n,k)^\beta,\hat d_X)$ with $(C_k(X_\beta),\hat d_X)$ via $v(\beta,\alpha) \mapsto w(\alpha)$. We conclude
$$H_k(C_X(n,\cdot)^\beta,\hat d_X) = H_k(X_\beta).$$

\begin{defn}\label{LandH} 
(1)  For all $0\le k,n\le d$,
$$L_X(n,k) := coker(C_X(n,k+1) \buildrel \hat d_X\over \to C_X(n,k))$$ 
with cochain differential $d_X:L_X(n,k) \to L_x(n+1,k)$ induced from the differential $d_X$ on $C_X(n,k)$.  
 
(2)  For all $0\le k,n\le d$, $H_X(n,k) = H^n(L_X(\cdot,k),d_X)$.
\end{defn}

We will need some tools to analyze the cohomology groups $H_X(n,k)$.  We note immediately that $L_X(n,k) = 0$ and $H_X(n,k) =0$ whenever $n<k$.  There are three simple lemmas which often suffice to calculate $H_X(n,k)$.

\begin{lemma}\label{shortex} For all $k\ge 1$ we have a short exact sequence of cochain:
$$0\to L_X(n,k) \to C_X(n,k-1) \to L_X(n,k-1) \to 0.$$  
\end{lemma}

\begin{proof}
As shown above, the homology of $\hat d_X$ on $C_X(n,\cdot)$ is the direct sum of the homology groups of the spaces $X_\beta$ as $\beta$ varies over the $n$-cells.  These are all contractible, hence $\hat d_X$ has homology only in degree $k=0$.  The result follows immediately. 
\end{proof}

We can always calculate the groups $H_X(n,k)$ for $k=0$.

\begin{lemma}\label{topology1}
$H_X(n,0) = H^n(X)$.  
\end{lemma}

\begin{proof}
 We have a well-defined surjective cochain map $f:C_X(n,0) \to C^n(X)$ given by 
$f(v(\alpha,\beta)) = v(\alpha)$. The kernel of this map is generated by the elements of the form
$v(\alpha,\beta) - v(\alpha,\beta')$ for any $n$-cell $e_\alpha$ and any two distinct 
0-cells $e_\beta$ and $e_\beta'$ contained in $X_\alpha$.    Since $X_\alpha$ is connected, its 1-skeleton $X_1$ is also connected and thus we can choose a sequence of 1-cells  $\gamma_1, \gamma_2, \ldots,\gamma_s$  in $X_\alpha$ which form a path
from $\beta$ to $\beta'$.  Let $\partial e_{\gamma_i} = \{e_{\beta_i},e_{\beta_{i+1}}\}$ with 
$\beta_1=\beta$ and $\beta_{r+1} =\beta'$. Then
$v(\alpha,\beta) - v(\alpha,\beta') = \sum_i v(\alpha,\beta_i) - v(\alpha,\beta_{i+1})$.
But $v(\alpha,\beta_i) - v(\alpha,\beta_{i+1}) = \pm \hat d_X(v(\alpha,\gamma_i))$.  
This shows that the kernel of $f:C_X(n,0) \to C^n(X)$ is the image of 
$\hat d_X:C_X(n,1) \to C_X(n,0)$.  We conclude that the induced map 
$L_X(n,0) \to C^n(X)$ is a cochain isomorphism, which proves the lemma.
\end{proof} 

\begin{rmk} All of the groups $C_X(n,k)$, $L_X(n,k)$ and $H_X(n,k)$ are defined with integral coefficients.  We could as easily give all of the same definitions with coefficients in an arbitrary 
field $F$ and we will denote the corresponding groups as $C_X(n,k;F)$, $L_X(n,k;F)$ and $H_X(n,k;F)$ respectively.  By Lemma \ref{topology1} we have $H_X(n,0;F) = H^n(X;F)$. 
\end{rmk}

Let $\tilde H$ denote
reduced cohomology (or homology).  Our last topological lemma deals with the nicest case, when the space $X$ is a manifold or a manifold with boundary.

\begin{lemma}\label{topology2}
Let $X$ be a connected regular CW-complex of dimension $d$ which is either 
a manifold or a manifold with boundary.  Then:

(1) For any $k$-cell $e_\alpha$,  $H^n(X,Y_\alpha) = 0$ for all $n<d$.

(2) Suppose further that $\tilde H^n(X;F) =0$ for $n<d$.  Then
for all $0\le k< n< d$, $H_X(n,k;F) = 0$ and for $0\le k < d$, 
$H_X(k,k;F) = F$.

(3) If $X = X_\beta$ then  $H_X(d,k;F) = 0$ for $k <d$.

\end{lemma}

\begin{proof}
Part (1).  For any $k$-cell $e_\alpha$ and any $n$-cell $e_\beta$ for which
$e_\alpha \subset D_\beta$, the space $\partial D_\beta \setminus e_\alpha$ is a deformation retract of
$D_\beta \setminus e_\alpha$.  These deformation retracts can be glued together to see that $Y_\alpha$ is a deformation retract of $X\setminus e_\alpha$.
If $e_\alpha$ is contained in the interior of $X$ then we get 
$H^n(X,Y_\alpha) = H^n(X,X\setminus e_\alpha) = \tilde H^n(S^d)$.  This is $0$ for
$n<d$.  If $e_\alpha$ is contained in the boundary of $X$, then
$H^n(X,Y_\alpha) = H^n(X,X\setminus e_\alpha) = 0$ for all $n$.  

Part (2).  To ease the exposition, we suppress $F$ from the notation.  
We have $C_X(n,k) = \oplus_\alpha C_X(n,k)_\alpha$.  For each
$\alpha$, $H^n(C(\cdot,k)_\alpha,d_X) = H^n(X,Y_\alpha)$.  Thus by (1) we have
$H^n(C_X(\cdot,k),d_X) = 0$ for all $0\le k\le n<d$.   We can now prove (2) by induction
on $k$.  For $k=0$ we get the result from the hypothesis on the cohomology of $X$ together with \ref{topology1}.  Assume the result for $k-1$ and consider the long exact sequence from \ref{shortex}.
The result follows immediately.

Part (3).  If $X=X_\beta$ is a $d$-disc, then $H^n(X) = H_X(n,0) = 0$ for all $0<n\le d$. Moreover, $H^d(X,Y_\alpha) = 0$ for all $\alpha$.  Hence the inductive proof used in (2) works when $n=d$. 
\end{proof}

\begin{rmk}\label{disc}
In the Lemma above we can consider the case when $X = X_\beta$ for some $d$-cell $e_\beta$, i.e.
$||X||$ is the manifold with boundary, $D_\beta^d$.  Then the lemma is true with
$F$ replaced by $\ZZ$ and furthermore $H_X(k,k) = \ZZ$ for all $0\le k\le d$.   
\end{rmk}

\section{Proofs of the main theorems.}

Throughout this section $X$ will denote a regular CW complex. We fix a field $F$. First 
we will consider the $F$-algebras $R(\bar P(X))$.  Let $\Gamma = \bar P(X)$.  
Since $\Gamma$ will vary, we write
$R_\Gamma(n,k)$ for the spaces $R(n,k)$ relative to $R(\Gamma)$, as defined in section 3.  

Recall that for any $k$-cell $e_\alpha^k$,  $X_\alpha$ is the subcomplex of $X$ whose total space is the disc  $||X_\alpha|| = D_\alpha^k$.  In particular we have 
$\Gamma_\alpha = \bar P(X)_\alpha = \bar P(X_\alpha) = \hat P(\partial X_\alpha)$. 

Choose two cells $\alpha< \beta$ in $X$ of dimensions $k$ and $n$ respectively and choose a path $\pi = \{\alpha = \alpha_0<\alpha_1<\ldots <\alpha_s = \beta\}$ in $\Pi(\beta, \alpha)$.  We set:
$$ r_\pi = r_{\alpha_s} r_{\alpha_{s-1}} \cdots r_{\alpha_0} \in R_\Gamma(n,k).$$
(Note: $s=n-k$.) We also define the {\it sign} of $\pi$ by the formula:
$$sgn(\pi) = d(\alpha_s,\alpha_{s-1})d(\alpha_{s-1},\alpha_{s-2}) \cdots d(\alpha_1,\alpha_0).$$

\begin{lemma}\label{pathindep}  Let $\alpha$ and $\beta$ be cells with
$\alpha <\beta$.  For any two paths $\pi$ and $\pi'$ in $\Pi(\beta,\alpha)$,
$sgn(\pi)r_{\pi} = sgn(\pi')r_{\pi'}$.

\end{lemma}

\begin{proof}
Suppose first that $dim(\beta) - dim(\alpha) = 2$. Then, since $\bar P(X)$ is thin, there are exactly two elements in $(\alpha,\beta)$, say $\gamma_1$ and $\gamma_2$.
We know $d(\beta,\gamma_1)d(\gamma_1,\alpha) + d(\beta,\gamma_2)d(\gamma_2,\alpha) = 0$.  This
just encodes the fact that $d_X^2 = 0$.  Let $S_\beta(1) = \{r_{\gamma_1}, 
r_{\gamma_2},\ldots, r_{\gamma_t}\}$.  Then for $j>2$, $\alpha \not < \gamma_j$ and
thus $r_{\gamma_j}r_\alpha = 0$.
Therefore, $0 = r_\beta r_\beta(1) r_\alpha = r_\beta r_{\gamma_1}r_\alpha
+ r_\beta r_{\gamma_2}r_\alpha$.  These formulas combine to show that
$sgn(\pi)r_\pi$ does not depend on which of the two paths in $\Pi(\beta,\alpha)$ we choose.

In general, 
let $\pi= \{\alpha = \alpha_0<\alpha_1<\ldots <\alpha_s = \beta\}$ and let
$\pi' = \{\alpha = \gamma_0<\gamma_1<\ldots <\gamma_s = \beta\}$.
By \ref{diamond}, we may assume that there is a unique $i$, $1\le i \le s-1$
such that $\alpha_i \ne \gamma_i$.  We may reduce then to the case
when $\alpha = \alpha_{i-1}$ and $\beta = \alpha_{i+1}$, which is covered by the previous paragraph
\end{proof}

Note:  we do not claim that $sgn(\pi) r_\pi \ne 0$.

In light of Lemma \ref{pathindep}, we have a well-defined linear map
$\Phi_X: C_X(n,k;F) \to R_\Gamma(n,k)$ given by extending linearly from:
$$\Phi_X(v(\beta,\alpha)) = sgn(\pi) r_\pi \ \hbox{ for any }\ \pi \in \Pi(\beta,\alpha).$$
This map is clearly surjective for all $n$ and $k$.  

Suppose now that $\beta$ is an $n$-cell, $\alpha$ is a $(k+1)$-cell and let
$S_\alpha(1) = \{\gamma_1,\ldots,\gamma_m\}$.  Choose $\pi\in \Pi(\beta,\alpha)$.
Then $\pi \cup \{\gamma_i\} \in \Pi(\beta,\gamma_i)$ for each $i$.  We calculate:
$$\begin{array}{ll} \Phi_X(\hat d_X(v(\beta,\alpha)))
& = \sum_i d(\alpha,\gamma_i) \Phi_X(v(\beta,\gamma_i)) =
   \sum_i d(\alpha,\gamma_i) sgn(\pi \cup \{\gamma_i\} )r_\pi r_{\gamma_i}\\ [6 pt]
&= \sum_i sgn(\pi)r_\pi r_{\gamma_i} = sgn(\pi)r_\pi r_\alpha(1) = 0.
\end{array}$$
We conclude that $\Phi_X$ factors through $coker(\hat d_X)$, to give us a well-defined
linear surjection $\Phi_X:L_X(n,k;F) \to R_\Gamma(n,k)$.

Fix an $n$-cell $\beta$ and let 
$\{\zeta_1,\ldots, \zeta_m\}$ be a list of all the $(n+1)$-cells $\zeta$ for which
$\beta<\zeta$.  For any $k$-cell 
$\alpha <\beta$, choose a path $\pi\in \Pi(\beta,\alpha)$. Again,
$\pi\cup \{\zeta_i\} \in \Pi(\zeta_i,\alpha)$.  We get:
$$\begin{array}{ll}
\Phi_X(d_X(v(\beta,\alpha))) 
&= \Phi_X( \sum_i d(\zeta_i,\beta)v(\zeta_i,\alpha))\\[6 pt]
&   = \sum_i d(\zeta_i,\beta)sgn(\pi \cup \{\zeta_i\})r_{\zeta_i} r_\pi\\[6 pt]
&= (\sum_i r_{\zeta_i})sgn(\pi)r_\pi\\[6 pt]
& = d_\Gamma \Phi_X(v(\beta,\alpha)).
\end{array}$$
Hence $\Phi_X: L_X(n,k) \to R_\Gamma(n,k)$ is a cochain epimorphism. 

\begin{thm}\label{easy}
Let $X$ be a regular CW complex and $\Gamma = \bar P(X)$.  Then:

(1)  $R(\Gamma)$ is a Koszul algebra.

(2)  $\Phi_X:L_X(n,k;F) \to R_\Gamma(n,k)$ is a cochain isomorphism (i.e. $\Phi_X$ is injective
for all $0\le k\le n\le dim(X)$).

\end{thm}

\begin{proof}
Let $d = dim(X)$. To simplify notation we will suppress tensoring with
$F$ and just write $L_X(n,k)$ rather than $L_X(n,k;F)$. 
We proceed by induction, assuming both (1) and (2) for any
regular cell complex of dimension smaller than $d$.  There is nothing to prove if
$d=0$, since in that case the algebra is a quadratic monomial algebra and all such algebras are Koszul (cf. \cite{PP}).  Moreover when $d=0$, part (2) is trivial. 

We consider first the special case that $X$ has a unique $d$-cell, $\beta$, i.e. $X = X_\beta$.  Consider:
$$ \begin{array}{cccccccccc}
0& \to & L_X(k,k) & \buildrel d_X\over \to &\ \ldots\  & \buildrel d_X \over \to & L_X(d-1,k) 
	& \buildrel d_X \over \to  & L_X(d,k) &\to 0 \\[6 pt]
&& \Phi_X \downarrow \hphantom{\Phi_X}
&&&&\Phi_X \downarrow \hphantom{\Phi_X}
&&\Phi_X \downarrow \hphantom{\Phi_X}\\[6 pt]
0& \to & R_\Gamma(k,k) & \buildrel d_\Gamma \over \to &\ \ldots\  & \buildrel d_\Gamma \over \to & R_\Gamma(d-1,k) 
	& \buildrel d_\Gamma \over \to  & R_\Gamma(d,k) & \buildrel d_\Gamma\over \to 0 \\
\end{array}$$
Since $\partial X$ is a $(d-1)$-sphere, (1) and (2) hold for $\partial X$.  
Let $\Omega = \bar P(\partial X)$.  The Koszul algebra 
$R(\Omega)$ is a subalgebra of $R(\Gamma)$ and we have  
$R_\Gamma(n,k) = R_\Omega(n,k)$ for all $0\le k\le n<d$.   Similarly we have $L_X(n,k) = L_{\partial X}(n,k)$ and 
$\Phi_X = \Phi_{\partial X}:L_{\partial X}(n,k)\to R_\Omega(n,k)$ whenever $0\le k\le n<d$.  

By induction, the maps $\Phi_{\partial X}$ are all isomorphisms and thus all but the
rightmost downward map in the diagram is an isomorphism.  The rightmost
downward map is an epimorphism.  Thus to prove (2) it suffices to see that
$\Phi_X:L_X(d,k) \to R_\Gamma(d,k)$ is injective. Let $K(d,k)$ be the kernel.  
Then we can consider $0\to K(d,k)\to 0$ as the cochain complex kernel of $\Phi_X$ and so we get a long exact sequence associated to the diagram above and that long exact sequence ends with:
$$\begin{array}{rl} 0\to H_X(d-1,k) \to & H^{d-1}(R_\Gamma(\cdot,k),d_\Gamma) \to K(d,k)\\[6 pt]
& \to H_X(d,k)  \to H^d(R_\Gamma(\cdot,k),d_\Gamma) \to 0.\\
\end{array}$$
By combining \ref{easycohom} and \ref{topology2} we see the following: If $k=d$ then the last two terms are $F$ and the first two terms are $0$, so $K(d,d)=0$.  
If $k=d-1$, then the first two terms are $F$ and the last two terms are $0$, so
$K(d,d-1)=0$.  If $k<d-1$ then the first two and last two terms are all $0$, so
$K(d,k) = 0$.  This proves (2). Moreoever, we see that $H^n(R_\Gamma(\cdot,k)d_\Gamma) =  H_X(n,k)$ for all $n$ and $k$.  Then Theorem \ref{R-cohomology} tells us $R(\Gamma)$ is Koszul, proving
(1). 

Now we can turn to the general case of a $d$-dimensional regular CW complex $X$.
For any cell $\beta$ of $X$, we know from the special case above
that $R(\Gamma_\beta)  = R(\bar P(X_\beta))$ is Koszul.  Hence by Corollary \ref{cyclic},
$R(\Gamma)$ is Koszul, proving (1).  

To prove (2) it suffices to prove $\Phi_X:L_X(n,k) \to R_\Gamma(n,k)$ is injective.  As above, we already know this inductively for $n<d$.  Let $\{\beta_1,\ldots,\beta_s\}$ be the set
of all $d$-cells of $X$.  By definition $C_X(d,k) = \oplus_iC_X(d,k)^{\beta_i}$ is a $\hat d_X$ chain complex decomposition. 
Thus $L_X(d,k)$ inherits a direct sum decomposition into subspaces
$L_X(d,k)^{\beta_i}$.  But this is nothing other than $L_{X_{\beta_i}}(d,k)$, which
by the special case is mapped injectively into $R_{\Gamma_{\beta_i}}(d,k)$.  By
\ref{rtideal}, $R_\Gamma(d,k) = \oplus_i R_{\Gamma_{\beta_i}}(d,k)$.  Thus $\Phi_X$ is injective, proving (2).  
\end{proof}

Theorem \ref{main1} follows from Theorem \ref{easy}.  

We now turn to the more subtle question of analyzing $R(\hat P(X))$.  Whenever 
$X$ is a $d$-dimensional regular complex that is connected by $(d-1)$-faces, we will
write $\hat \Gamma$ for  $\hat P(X)$.  The following theorem gives necessary and sufficient conditions for
$R(\hat \Gamma)$ to be Koszul.

\begin{thm}\label{hard}  Let $X$ be a $d$-dimensional regular CW complex and assume
$X$ is connected by $(d-1)$-faces.  Then $R(\hat \Gamma)$ is Koszul if and only if
$H_X(n,k;F) = 0$ for all $0\le k<n<d$.    
\end{thm}

\begin{proof}  Recall that we denote the unique maximal element $\hat \Gamma$ by $\bar 1$.  
The rank of $\bar 1$ is $d+2$.
By \ref{easy}, the subalgebra $R(\Gamma)$ is Koszul and in particular, for every
cell $\beta \in (\bar 0,\bar 1)$, $R(\hat \Gamma_\beta) = R(\bar P(X_\beta))$ is Koszul.  Hence we can combine
Lemma \ref{easycohom} with Theorem \ref{R-cohomology} to see that $R(\hat \Gamma)$ is Koszul if and only if
$H^n(R_{\hat \Gamma}(\cdot,k),d_{\hat \Gamma}) = 0$ for all $0\le k <n<d$.   
But for $0\le k\le n\le d$,   
$R_{\hat \Gamma}(n,k) = R_\Gamma(n,k)$ and so these cochain complexes have the same cohomology
for all $0\le k\le n< d$.  Applying part (2) of \ref{easy}, we see that
$R(\hat \Gamma)$ is Koszul if and only if $H_X(n,k;F) = 0$ for all $0\le k<n<d$, 
as required.  
\end{proof}

\begin{rmk} In the argument above, $L_X(n,k;F)$ and $R_{\hat \Gamma(n,k)}$ are isomorphic for $n<d+1$.  There is no term of the form $L_X(d+1,k;F)$ in the 
$L_X(n,k;F)$ cochain,  but there is an $R_{\hat \Gamma}(d+1,k)$ term in the
$R_{\hat \Gamma}(n,k)$ cochain.  This explains why $H_X(d,k;F)$ can be nonzero even though, by \ref{easycohom}, $H^d(R_{\hat \Gamma}(\cdot,k),d_\Gamma) = 0$ for $k<d$.
\end{rmk}

We have several immediate corollaries.

\begin{cor}\label{bouquet}
If $R(\hat \Gamma)$ is Koszul, then $H^n(X;F) = 0$ for $0<n<d$.  
\end{cor}

\begin{proof}  $H^n(X;F) = H_X(n,0;F)$.
\end{proof}

\begin{cor}  If $X$ is a 2-dimensional regular complex and is connected by 1-faces, then
$R(\hat \Gamma)$ is Koszul if and only if $H^1(X;F) = 0$.
\end{cor}

\begin{proof}
According to \ref{hard}, the only obstruction to the Koszul property is $H_X(1,0;F) = H^1(X;F)$.
\end{proof}

When $X$ is a manifold or a manifold with boundary we combine \ref{hard} with Lemma \ref{topology2} to get a partial converse to \ref{bouquet}.

\begin{cor}
If $X$ is a $d$-dimensional connected manifold or manifold with boundary then
$R(\hat \Gamma)$ is Koszul if and only if $H^n(X;F) = 0$ for all $0<n<d$. 
\end{cor}

We can also restate \ref{hard} in more familiar topological terms.

\begin{cor}\label{topology3}
Let $X$ be a $d$-dimensional regular complex and assume
$X$ is connected by $(d-1)$-faces.  Then $R(\hat \Gamma)$ is Koszul if and only if $H^n(X;F) = 0$ for 
$0<n<d$ and for any cell $e_\alpha$ of $X$, $H^n(X,Y_\alpha;F) = 0$ for all $n<d$. 
\end{cor}

\begin{proof}
As usual we supress $F$. The proof of Lemma \ref{topology2} shows how to lift the vanishing of the cohomology groups
$H^n(X,Y_\alpha)$ and $\tilde H^n(X)$ for $n<d$ to the vanishing of the cohomology
groups $H_X(n,k)$ for $0\le k <n<d$.  This proves the if part of the corollary.  Conversely, if we know
that $H_X(n,k) =0$ for $0\le k<n<d$, then the long exact cohomology sequence used in \ref{topology2}
proves that $H^n(C_X(\cdot,k),d_X) = 0$ for $0\le k<n<d$.  But this cohomology is the direct sum
of the cohomology groups $H^n(X,Y_\alpha)$. 
\end{proof}

We end the paper by using \ref{topology3} to analyze an example which will prove that the converse of
Corollary \ref{bouquet} is not true.  This example amply illustrates the strength of our topological tools.

\begin{example}\label{singular}  Let $X$ be the 3-dimensional regular CW complex obtained as follows.  Begin with two simplicial 3-cells (solid tetrahedra) $A_1$ and 
$A_2$.  On each $A_i$, choose one 2-simplex $\beta_i$ and one 1-simplex $\alpha_i$ that is not an edge of $\beta_i$.  Then $\alpha_i$ and $\beta_i$ must have a common boundary point $\zeta_i$.  Glue $\beta_1$ to $\beta_2$ in such a way that $\zeta_1$ and $\zeta_2$ are identified.  Call the new resulting face $\beta$.  Then glue $\alpha_1$ to $\alpha_2$ and call the resulting
edge $\alpha$. An attempt to draw this construction is given below. The resulting solid, $X$, has 2 3-cells, 7 2-cells, 8 1-cells and 4 0-cells.  The Hasse diagram of 
$\hat P(X)$ is also shown below.  In this diagram, $\beta$ is $B_4$,
$\alpha$ is $C_4$ and the 0-cell common to the boundary of $\beta$ and $\alpha$, indicated by $\zeta$ in the picture, is $D_3$.  The regular CW complex $X$ is not simplicial but it is connected by 2-faces. 

  \begin{center}
     \includegraphics[width=.8 \textwidth]{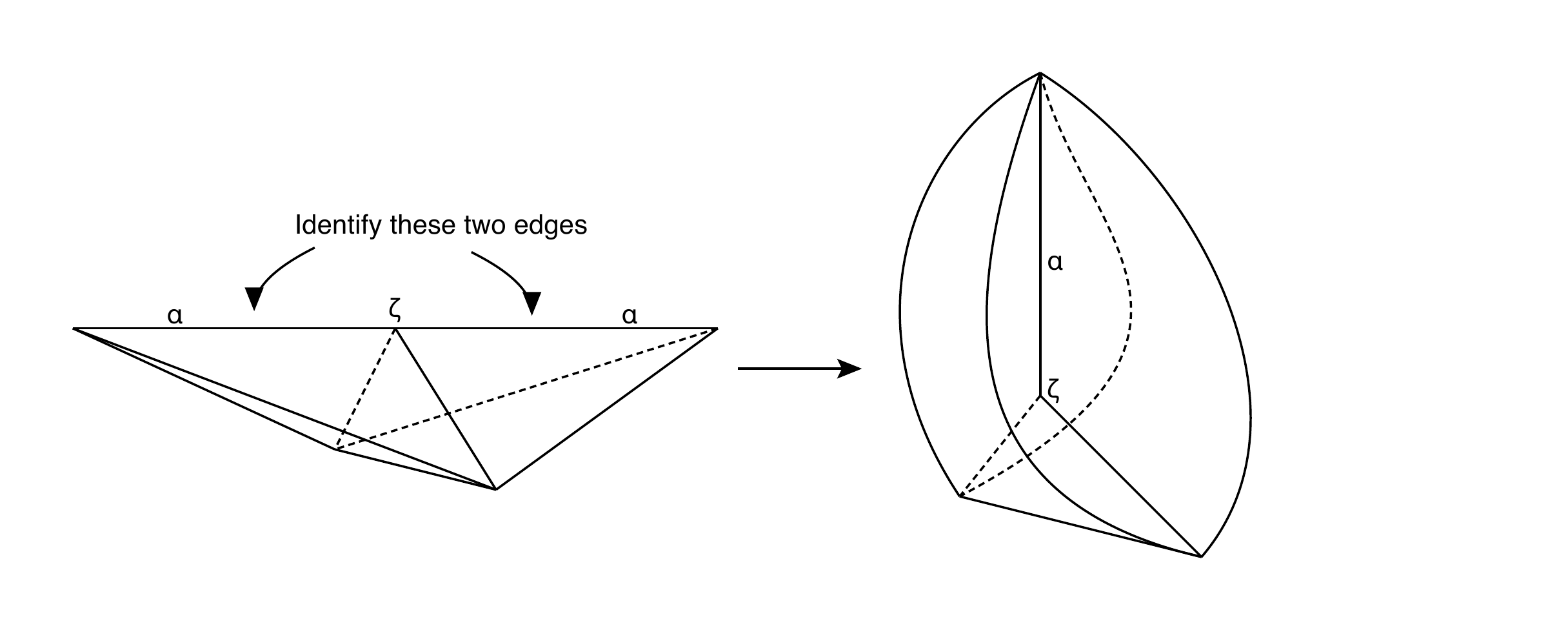}
  \end{center}
 
\xymatrix{
 & & & & X=\bar 1 \ar[dl] \ar[dr] \cr
 & & & A_1 \ar[dll] \ar[dl] \ar[d] \ar[dr] & &A_2 \ar[dl] \ar[d] \ar[dr] \ar[drr] \cr
 & B_1 \ar[ddl] \ar[dd] \ar[ddrrr] & 
 B_2  \ar[ddll] \ar[ddr] \ar[dd] & 
 B_3  \ar[ddll] \ar[dd] \ar[ddrr] & 
 B_4 \ar[ddll] \ar[dd] \ar[ddr] & 
 B_5 \ar[ddlll] \ar[ddll]  \ar[ddr] & 
 B_6 \ar[ddl] \ar[ddlll] \ar[ddr]& 
 B_7 \ar[ddlll] \ar[ddl] \ar[dd] \cr
 \cr
 C_1 \ar[ddrr] \ar[ddrrr] & 
 C_2 \ar[ddr] \ar[ddrrrr] & 
 C_3 \ar[ddr] \ar[ddrr] & 
 C_4 \ar[ddl] \ar[ddr] & 
 C_5 \ar[ddl] \ar[ddr] & 
 C_6 \ar[ddl] \ar[dd] & 
 C_7 \ar[ddlll] \ar[ddllll] & 
 C_8 \ar[ddlllll] \ar[ddll]\cr
 \cr
 & & D_1 \ar[dr] & D_2 \ar[d] & D_3 \ar[dl] & D_4 \ar[dll] \cr
& & & \overline{0}
 }
\end{example}

\begin{thm}\label{theexample}  Let $X$ be the 3-dimensional regular CW complex described in \ref{singular}.  Then
$H^n(X;F) = 0$ for all $0<n<3$, but $R(\hat \Gamma)$ is not Koszul.
\end{thm}

\begin{proof}
The cohomology statement is trivial, since $X$ is contractible. 
Consider the 1-cell $e_\alpha$ marked in the figure.  Then $Y_\alpha$ is homotopic to $S^1$.  Therefore 
$H^2(X,Y_\alpha;F) \ne 0$.  By \ref{topology3}, $R(\hat \Gamma)$ is not Koszul.
\end{proof}

\bibliographystyle{amsplain}
\bibliography{bibliog}

\end{document}